\documentclass{amsart}
\usepackage{amsthm,amsmath,amssymb}
\usepackage[pdftex]{graphicx}
\usepackage{braket}
\usepackage{color}

\title[Topologically slice knots that are not smoothly slice in any definite 4-manifold]{Topologically slice knots that are not smoothly slice in any definite 4-manifold}
\author{Kouki Sato}
\date{}

\newtheorem{thm}{Theorem}
\newtheorem{prop}{Proposition}
\newtheorem{lem}{Lemma}

\theoremstyle{definition}

\newtheorem*{remark}{Remark}
\newtheorem*{acknowledge}{Acknowledgements}

\DeclareMathOperator{\Wh}{Wh}

\begin{document}
\maketitle

\begin{abstract}
We prove that there exist infinitely many topologically slice knots 
which cannot bound a smooth null-homologous disk
in any definite 4-manifold. Furthermore, we show that we can take such knots so that they are linearly independent in the the knot concordance group.    
\end{abstract}

\section{Introduction}

A knot $K$ in $S^3$ is called {\it smoothly slice} ({\it topologically slice})
if $K$ bounds a smooth disk (resp.\ topologically locally flat disk) in $B^4$.
While any smoothly slice knot is obviously topologically slice,
it has been known that there exist infinitely many topologically slice knots that are not smoothly slice (for instance, see \cite{endo, gompf}).
The purpose of this paper is to prove that there exist infinitely many topologically slice knots
which cannot bound a null-homologous smooth disk not only in $B^4$ but also in any 4-manifold with definite intersection form. 

For a 4-manifold $V$ with boundary $S^3$, we call a knot $K$ in $S^3$ {\it 
smoothly slice in} $V$ 
if $K$ bounds a smooth disk $D$ in $V$  such that 
$[D, \partial D] = 0 \in H_2(V,\partial V; \mathbb{Z})$.
We call a 4-manifold $V$ {\it definite} if the intersection form of $V$ is either positive definite or negative definite. We denote the smooth knot concordance group by $\mathcal{C}$.
Then our main theorem is stated as follows.

\begin{thm}
\label{thm1}
There exist infinitely many topologically slice knots which are not slice in any definite 4-manifold. Furthermore, we can take such knots so that they are linearly independent in 
$\mathcal{C}$.
\end{thm}

In order to prove Theorem \ref{thm1}, we use the Heegaard Floer $\tau$-invariant and the $V_k$-invariants  defined in \cite{ni-wu}. In particular, by combining 
Wu's cabling formula \cite{wu} and 
Bodn\'{a}r-Celoria-Golla's connected sum formula \cite{bodnar-celoria-golla} for $V_k$-invariants, we prove the following proposition.
Here we denote the mirror image of  a knot $K$ by $K^*$, the $(n,1)$-cable of $K$ 
by $K_{n,1}$ and the connected sum of two knots $K$ and $J$ by $K \# J$.
\begin{prop}
\label{prop2}
Let $K$ and $J$ be knots. If $V_0(K) > V_0(J)$ and $\tau(K), \tau(J)>0$,
then for any positive integer $n$ with $\tau(K) < n \cdot \tau(J)$, the knot $K\#(J_{n,1})^*$ is not slice in any definite 4-manifold.
\end{prop} 

Note that if both $K$ and $J$ are topologically slice, then $K\#(J_{n,1})^*$ is also topologically slice for any $n \in \mathbb{Z}\setminus \{0\}$. Furthermore, it follows from 
\cite[Proposition 6.1]{hedden-kim-livingston} and 
\cite[Theorem B.1]{hedden-kim-livingston} 
that for any $m \in \mathbb{Z}_{>0}$, there exists a topologically slice knot $K_m$ with $V_0(K_m)=m$. 
Hence by taking $K_l \# ((K_m)_{n,1})^*$ so that $l>m$ and $n$ is sufficiently large,
we immediately obtain infinitely many topologically slice knots which are not slice in any definite 4-manifold. Our proof of the linear independence of these topologically slice knots relies on Kim-Park's recent result \cite{kim-park}.

The problem of smooth sliceness leads to the notion of the {\it kinkiness} of knots, as defined by Gompf \cite{gompf}.
Let $K$ be a knot in $S^3= \partial B^4$, and
consider all self-transverse immersed disks in $B^4$ with boundary $K$.
Then we define
$k_+(K)$ (resp.\ $k_-(K)$) to be the minimal number of positive (resp.\ negative) 
self-intersection points occurring in such a disk.
Gompf proved in \cite{gompf} that for any $n \in \mathbb{Z}_{>0}$,
there exists a topologically slice knot $K$ such that $(k_+(K),k_-(K))=(0,n)$.
On the other hand, as far as the author knows,
whether there exist topologically slice knots which satisfy $k_+>0$ and $k_->0$ 
remained so far unsolved. 
In this paper, we give an affirmative answer to the question.

\begin{thm}
\label{thm2}
For any $m,n \in \mathbb{Z}_{>0}$,  there exist infinitely many topologically slice knots with $k_+ \geq m$ and $k_- \geq n$.
\end{thm}

\begin{acknowledge}
The author was supported by JSPS KAKENHI Grant Number 15J10597.
The author would like to thank his supervisor, Tam\'{a}s K\'{a}lm\'{a}n
for his encouragement and useful comments.
The author also would like to thank 
Wenzhao Chen, Marco Golla, Matthew Hedden and Jennifer Hom
for their stimulating discussions.
\end{acknowledge}

\section{Preliminaries}

In this section, 
we recall some knot concordance invariants derived from Heegaard Floer homology theory, and show that they give obstructions to sliceness of knots in definite 4-manifolds. 

\subsection{Correction terms and $d_1$-invariant}

Ozsv\'{a}th and Szab\'{o} \cite{ozsvath-szabo} introduced a $\mathbb{Q}$-valued invariant
$d$ (called the {\it correction term})
for rational homology 3-spheres endowed with a Spin$^c$ structure.
In particular, since any integer homology 3-sphere $Y$ has a unique Spin$^c$ structure,
we may denote the correction term simply by $d(Y)$ in this case.
Furthermore, 
we note  that for any integer homology 3-sphere $Y$, $d(Y)$ is an even integer. 

Let $S^3_1(K)$ denote the $1$-surgery along  a knot $K$ in $S^3$.
Then $S^3_1(K)$ is an integer homology 3-sphere,
and hence we can define the {\it $d_1$-invariant} of $K$ as $d_1(K):=d(S^3_1(K))$.
It is known that $d_1(K)$ is a knot concordance invariant of $K$.
For details, see \cite{peters}. 
Here we show that the $d_1$-invariant gives an obstruction to sliceness in negative definite 4-manifolds.
\begin{lem}
\label{lem1}
If a knot $K$ is smoothly slice in some negative definite 4-manifold,
then we have $d_1(K) = 0$.
\end{lem}

\begin{proof}
It is proved in \cite{peters} that $d_1(K) \leq 0$ for any knot $K$.
Hence we only need to show that $d_1(K) \geq 0$. 

Suppose that $K$ is slice in a negative definite 4-manifold $V$. Then there exists a
properly embedded null-homologous disk $D$ in $V$ with boundary $K$.
By attaching a $(+1)$-framed 2-handle $h^2$ along $K$, and gluing $D$ with
the core of $h^2$, we obtain an embedded 2-sphere $S$ in $W := V \cup h^2$
with self-intersection $+1$. This implies that there exists a 4-manifold $W'$ 
with boundary $S^3_1(K)$ such that $W = W' \# \mathbb{C}P^2$. 
Note that $\partial W' = \partial W = S^3_1(K)$.
Since the number of positive eigenvalues of the intersection form of $W$ is one,
the intersection form of $W'$ must be negative definite.
Now we use the following theorem.
\begin{thm}[Ozsv\'{a}th-Szab\'{o}, \text{\cite[Corollary 9.8]{ozsvath-szabo}}]
\label{thm d_1}
If $Y$ is an integer homology 3-sphere with $d(Y)<0$, then there is
no negative definite 4-manifold $X$ with $\partial X = Y$.
\end{thm}
By Theorem \ref{thm d_1} and the existence of $W'$, we have $d_1(K) = d(S^3_1(K)) \geq 0$.
\end{proof}

\subsection{$\tau$-invariant, $V_k$-invariant and $\nu^+$-invariant}
The {\it $\tau$-invariant $\tau$} is
a famous knot concordance invariant defined by Ozsv\'{a}th-Szab\'{o} 
\cite{ozsvath-szabo3} and Rasmussen \cite{rasmussen}. It is known that
$\tau$ is a group homomorphism from $\mathcal{C}$ to $\mathbb{Z}$, while $d_1$ is
not a homomorphism.

The {\it $V_k$-invariant } is a family of $\mathbb{Z}_{\geq 0}$-valued knot concordance invariants $\{V_k(K)\}_{k \geq 0}$ defined by Ni and Wu \cite{ni-wu}. 
In particular, $\nu^+(K) := \min \{ k \geq 0 \mid V_k(K) =0 \}$ is known as the {\it $\nu^+$-invariant} \cite{hom-wu}. 
It is proved in \cite{hom-wu} that for any knot $K$, the inequality $\tau(K) \leq \nu^+(K)$ holds.

In \cite{ni-wu}, Yi and Ni prove that the set $\{V_k(K)\}_{k \geq 0}$ determines all correction terms of $p/q$-surgeries along $K$ for any coprime $p,q>0$. 
Let $S^3_{p/q}(K)$ denote the $p/q$-surgery along $K$.
Note that there is a canonical identification between the set of Spin$^c$ structures over $S^3_{p/q}(K)$ and $\{ i \mid 0 \leq i \leq p-1 \}$.
(This identification can be made explicit by the procedure in \cite[Section 4, Section 7]{ozsvath-szabo4}.)

\begin{prop}[Ni-Wu, \text{\cite[Proposition 1.6]{ni-wu}}]
\label{prop ni-wu}
Suppose $p,q>0$, and fix $0 \leq i \leq p-1$. Then
$$
d(S^3_{p/q}(K),i) = d(S^3_{p/q}(O),i) - 2 
\max \left\{ V_{\lfloor \frac{i}{q} \rfloor}(K), V_{\lfloor \frac{p+q -1-i}{q} \rfloor}(K) \right\},
$$
where $O$ denotes the unknot and $\lfloor \cdot \rfloor$ is the floor function.
\end{prop}
As a corollary, the following lemma holds.
(Note that $\{ V_k(K) \}_{k \geq 0}$ satisfy the inequalities $V_k(K) -1 \leq V_{k+1}(K) \leq V_{k}(K)$ for each $k \geq 0$.)
\begin{lem}
\label{lem2}
For any knot $K$, we have $d_1 (K) = -2 V_0(K)$.
\end{lem}

Here we show that the $\tau$-invariant also gives an obstruction to sliceness in negative definite 4-manifolds. 

\begin{lem}
\label{lem3}
If a knot $K$ is slice in some negative definite 4-manifold,
then we have $\tau(K) \leq 0$.
\end{lem}

\begin{proof}
Suppose that $K$ is slice in some negative definite 4-manifold. Then by Lemma \ref{lem1},
we have $d_1(K)=0$. By Lemma \ref{lem2}, this implies that $V_0(K)=0$ and $\nu^+(K)=0$.
Hence we have $\tau(K) \leq \nu^+(K) = 0$.
\end{proof}

By combining Lemma \ref{lem2} and Lemma \ref{lem3}, we obtain the following obstruction to sliceness in definite 4-manifolds.

\begin{prop}
\label{prop4}
Let $K$ be a knot. If $d_1(K) \neq 0$ and $\tau(K) < 0$, then $K$ is not slice in any definite 4-manifold.
\end{prop}

\begin{proof}
It immediately follows from Lemma \ref{lem1} that $K$ is not slice in any negative definite 4-manifold. Suppose that $K$ is slice in a positive definite 4-manifold $V$. Then by reversing the orientation of $V$, we obtain a slice disk in $-V$ with boundary $K^*$. 
Since $-V$ is negative definite and $\tau$ is a group homomorphism from $\mathcal{C}$ to $\mathbb{Z}$,  Lemma \ref{lem3} implies
$$
\tau(K) = - \tau(K^*) \geq 0.
$$
This contradicts the assumption $\tau(K) <0$.
\end{proof}

\subsection{Some formulas for $V_k$-invariants}

In this subsection, we recall Wu's cabling formula and Bodn\'{a}r-Celoria-Golla's connected sum formula for $V_k$-invariants.
Since the $(p,1)$-cable and the connected sum of topologically slice knots  are also topologically slice, we can estimate the $V_k$-invariants
of various topologically slice knots by using these formulas.  

We first recall Wu's cabling formula for $V_k$. 
For coprime integers $p, q >0$, let $T_{p,q}$ denote the $(p,q)$-torus knot and $K_{p,q}$ the $(p,q)$-cable of a knot $K$.
We define a map 
$$
\phi_{p,q} : \Big\{ i \  \Big| \  0\leq i \leq \frac{pq}{2} \Big\} \to 
\Big\{ i \  \Big| \  0 \leq i \leq q-1 \Big\}
$$
by
$$
\phi_{p,q}(i) \equiv i - \frac{(p-1)(q-1)}{2} \mod q.
$$

\begin{prop}[Wu, \text{\cite[Lemma 5.1]{wu}}]
\label{prop wu}
Given $p,q >0$ and $0 \leq i \leq \frac{pq}{2}$, we have
$$
V_i(K_{p,q}) = V_i(T_{p,q}) + \max 
\left\{ V_{\lfloor \frac{\phi_{p,q}(i)}{p} \rfloor}(K), 
V_{\lfloor \frac{p+q-1-\phi_{p,q}(i)}{p} \rfloor}(K) \right\}.
$$
\end{prop}
(Here we corrected a small mistake in Wu's paper; there should be no factor of 2 in front of the maximum. His proof clearly establishes the formula above.)
If we consider the case where $q=1$, then we have $\phi_{p,1}(i) = 0$ for any  $0 \leq i \leq \frac{p}{2}$.
Hence Proposition \ref{prop wu} gives the following lemma.
\begin{lem}
\label{q=1}
Given $p>0$ and $0 \leq i \leq \frac{p}{2}$, we have
$$
V_i(K_{p,1}) = V_0(K).
$$
\end{lem}

Next we recall Bodn\'{a}r-Celoria-Golla's connected sum formula for $V_k$.

\begin{prop}[Bodn\'{a}r-Celoria-Golla, \text{\cite[Proposition 6.1]{bodnar-celoria-golla}}]
\label{prop BCG}
For any two knots $K$ and $J$ and any $m, n \in \mathbb{Z}_{\geq 0}$, we have
$$
V_{m+n}(K \# J) \leq V_m(K) + V_n(J).
$$
\end{prop}

Here I would like to point out that in the proof of this proposition for the cases where $m=0$ or $n=0$, the authors of \cite{bodnar-celoria-golla} apply Ni-Wu's formula (i.e, Proposition \ref{prop ni-wu} in the present paper) to $S^3_0(K)$ and $S^3_0(J)$ without any comment, although Ni-Wu's formula is only proved for $S^3_{p/q}(K)$ with coprime $p,q>0$. 
For completeness, we give a proof of Proposition \ref{prop BCG} for the cases where $m=0$ or $n=0$.

\def\proofname{Proof of Proposition \ref{prop BCG}}
\begin{proof}
We first consider the case where $m=0$ and $n >0$.
Let $W$ be a 4-manifold represented by the relative diagram in Figure \ref{cob1},
which is an oriented cobordism from $S^3_2(K) \# S^3_{2n}(J)$ to $S^3_{2n+2}(K \# J)$.
(For details of relative diagrams, see \cite{gompf-stipsicz}.)  
Furthermore, we define $X$ to be a 4-manifold represented by the diagram in Figure \ref{cob1} with all brackets deleted, and $\widetilde{X}$ to be the closure of $X \setminus W$. It follows from elementary algebraic topology that $H_2(W;\mathbb{Z}) \cong \mathbb{Z}$ and 
$\sigma(W) = \sigma(X) -\sigma (\widetilde{X}) = -1$, hence $W$ is a negative definite.
We use the following lemma, which immediately follows from \cite[Theorem 9.6]{ozsvath-szabo}. 
\begin{lem}[Ozsv\'{a}th-Szab\'{o}]
\label{lem o-s}
Let $Y_1$ and $Y_2$ be rational homology 3-spheres and
$W$ a negative definite cobordism from $Y_1$ to $Y_2$.
Then for any Spin$^c$ structure $\frak{s}$ over $W$, we have the inequality
$$
c_1(\frak{s})^2 + \beta_2(W) \leq 4d(Y_2, \frak{s}|_{Y_2}) - 4d(Y_1, \frak{s}|_{Y_1}).
$$
\end{lem}

\begin{figure}[htbp]
\hspace{-1mm}
\includegraphics[scale = 0.7]{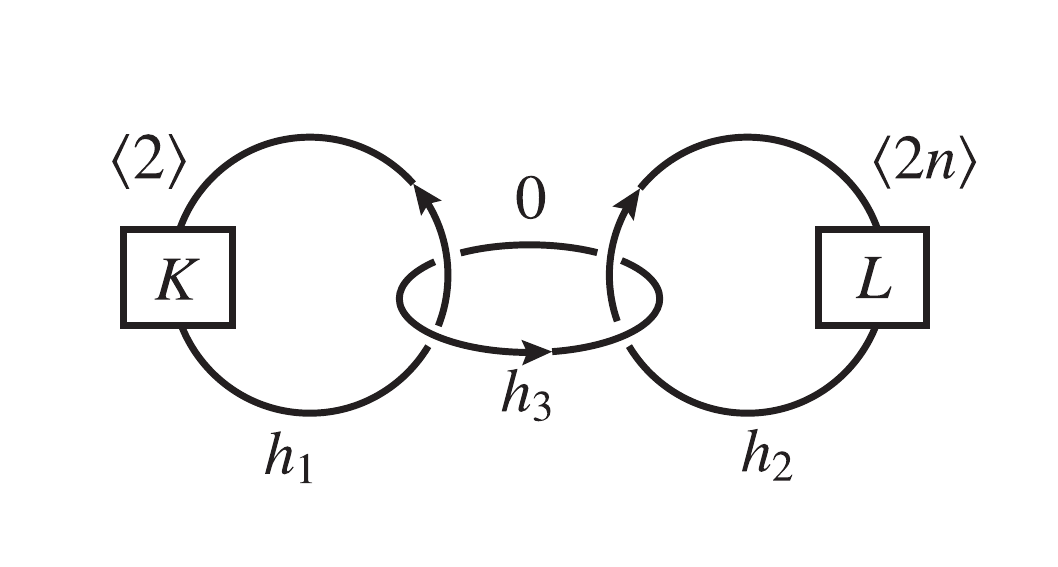}
\caption{\label{cob1}}
\hspace{-1mm}
\includegraphics[scale = 0.7]{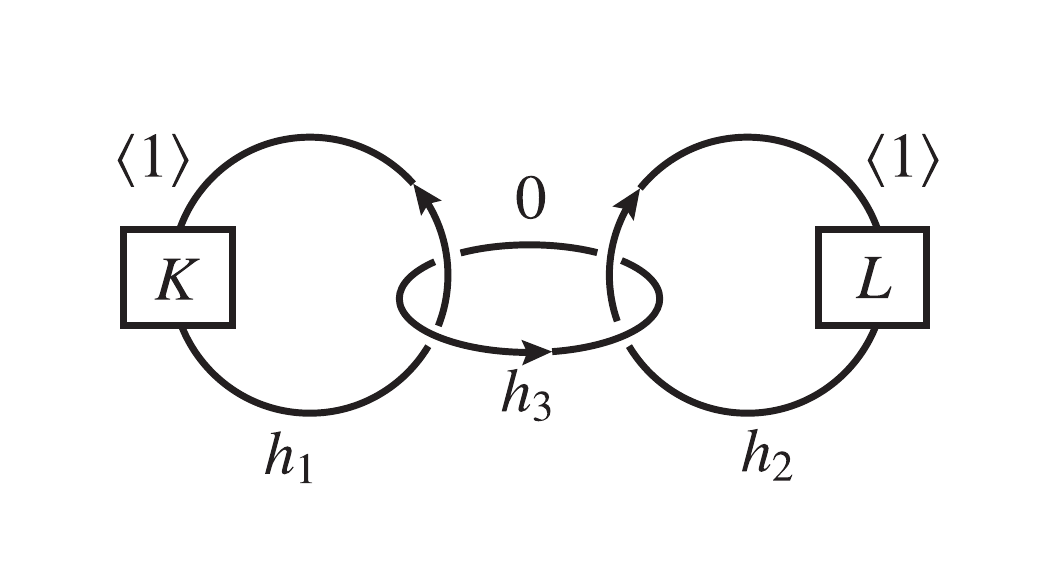}
\caption{\label{cob2}}
\end{figure}

To apply Lemma \ref{lem o-s} to $W$, we take a Spin$^c$ structure over $W$ as follows.
By choosing the order and the orientation of generators of $H_2(X; \mathbb{Z})$ as shown in Figure \ref{cob1}, we have a representation matrix
$$
\left(
\begin{array}{ccc}
2 & 0 & 1\\
0 & 2n & 1\\
1 & 1 & 0
\end{array}
\right)
$$
for the intersection form 
$Q_X: H_2(X;\mathbb{Z}) \times H_2(X;\mathbb{Z}) \to \mathbb{Z}$.
Let $h^*_i$ denote the cocore of $h_i$ ($i=1,2,3$).
We take a Spin$^c$ structure $\frak{s}$ over $X$ such that 
$PD(c_1(\frak{s})) = (-2,0,0) \in H_2(X, \partial X; \mathbb{Z})$ with respect to the basis $\{h^*_1, h^*_2, h^*_3 \}$.
For the restriction of $\frak{s}$ to $W$, we have the inequality
\begin{eqnarray}
\label{ineq1}
c_1(\frak{s}|_W)^2 + 1 & \leq & 
4d(S^3_{2}(K\#J), \frak{s}|_{S^3_{2n+2}(K\#J)})\\
\ & \ & -4d(S^3_{2}(K), \frak{s}|_{S^3_{2}(K)}) -4d(S^3_{2n}(J), \frak{s}|_{S^3_{2n}(J)}).
\nonumber
\end{eqnarray}
Here we compute $c_1(\frak{s}|_W)^2$.
Let $Q_{\widetilde{X}}$ and $Q_W$
denote the intersection form of $\widetilde{X}$ and $W$ respectively.
Since the inclusion maps induce an isomorphism 
$H_2(X;\mathbb{Q}) \cong H_2(\widetilde{X};\mathbb{Q}) \oplus H_2(W;\mathbb{Q})$ 
and $Q_X$ decomposes as $Q_{\widetilde{X}} \oplus Q_{W}$ over $\mathbb{Q}$, we have
$$
c_1(\frak{s})^2 = c_1(\frak{s}|_{\widetilde{X}})^2 + c_1 (\frak{s}|_W)^2. 
$$
Furthermore, it is not hard to check that $c_1(\frak{s})^2 = 2/(n+1)$  and $c_1(\frak{s}|_{\widetilde{X}})^2 = 2$.
Hence we have $c_1(\frak{s}|_W)^2 = -2n/(n+1)$.
Next we consider which integers correspond to $\frak{s}|_{S^3_2(K)}$, 
$\frak{s}|_{S^3_{2n}(J)}$ and $\frak{s}|_{S^3_{2n+2}(K \# J)}$ respectively.
For a knot $L$ and $k \in \mathbb{Z}$, let $X_k(L)$ denote a 4-manifold obtained by attaching a single 2-handle $h_{L}^2$ to $B^4$ along a knot $K\#L \subset \partial B^4$ with framing $k$. In addition, let $F_{L}$ denote a closed surface obtained by gluing the core of $h^2_{L}$ 
to a Seifert surface for $L$.
Then we note that $\widetilde{X}$ and
$X$ are diffeomorphic to $X_2(K) \natural X_{2n}(J)$ and
$X_{2n+2}(K \# J) \# S^2 \times S^2$ respectively
(where $\natural$ denotes boundary connected sum).
Furthermore, we can see that $[F_{K}]=(1,0,0)$, $[F_{J}]=(0,1,0)$ and
$[F_{K \# J}]=(1,-1,2n)$
with respect to the basis $\{ [h_1], [h_2], [h_3]\}$.
Hence we have 
$$
\langle c_1 (\frak{s}), [F_{K}] \rangle = -2,
\langle c_1 (\frak{s}), [F_{J}] \rangle = 0
\text{ and }
\langle c_1 (\frak{s}), [F_{K \# J}] \rangle = -2.
$$
This implies that 
$\frak{s}|_{S^3_2(K)}$, 
$\frak{s}|_{S^3_{2n}(J)}$
and $\frak{s}|_{S^3_{2n+2}(K\# J)}$ are 
identified with $0$, $n$ and $n$ respectively.
Now we can reduce the inequality $(\ref{ineq1})$ to 
\begin{eqnarray*}
-\frac{2n}{n+1} + 1 &\leq& 4d(S^3_{2n+2}(O),n) - 4d(S^3_{2}(O),0) - 4d(S^3_{2n}(O),n)\\
\ &\ & - 8V_{n}(K\#J) + 8V_0(K)  + 8V_{n}(J).
\end{eqnarray*}

Since $4d(S^3_{2n+2}(O),n) = 1-2n/(n+1)$, $4d(S^3_{2}(O),0) = 1$ and $4d(S^3_{2n}(O),n)= -1$,
we have the desired inequality 
$V_n(K\#J) \leq V_0(K) + V_n(J)$.

For the case where $m=n=0$, it suffices to apply the above argument after replacing Figure \ref{cob1} with Figure \ref{cob2}, taking a Spin$^c$ structure $\frak{s}$ over $X$
such that $PD(c_1(\frak{s})) = (1,1,2)$. 
\end{proof}

In this paper, we only need Proposition \ref{prop BCG} in the case where $m=n=0$, which is stated as follows.

\begin{prop}
\label{m=n=0}
For any two knots $K$ and $J$, we have
$$
V_{0}(K \# J) \leq V_0(K) + V_0(J).
$$
\end{prop}

We can use Proposition \ref{m=n=0} to give a lower bound for $V_0$ of the connected sum of two knots as well. In particular, we have the following lemma.
\begin{lem}
\label{lem4}
For any two knots $K$ and $J$, we have
$$
V_0 (K \# J^*) \geq V_0(K) - V_0(J).
$$
\end{lem}

\def\proofname{Proof}
\begin{proof}
For the inequality in Proposition \ref{m=n=0}, by replacing $K$with $K \# J^*$, we have
$$
V_0(K \# J^* \# J) \leq V_0(K \# J^*) + V_0(J).
$$
Since $K \# J^* \# J$ is concordant to $K$, we have $V_0 (K \# J^* \# J) = V_0(K)$.
This completes the proof.
\end{proof}

\section{Proof of the main theorems}
In this section, we prove Proposition \ref{prop2}, Theorem \ref{thm1} and Theorem \ref{thm2}.
We first prove Proposition \ref{prop2} and Theorem \ref{thm1}.

\def\proofname{Proof of Proposition \ref{prop2}}
\begin{proof}
Suppose that two given knots $K$ and $J$ satisfy $V_0(K) > V_0(J)$ and $\tau(K), \tau(J) > 0$. Fix a positive integer $n$ with $\tau(K) < n\cdot \tau(J)$.
Then Lemma \ref{q=1} and Lemma \ref{lem4} imply that
$$
V_0(K\# (J_{n,1})^*) \geq V_0(K) -V_0(J_{n,1})= V_0(K) - V_0(J)>0.
$$
Hence by Lemma \ref{lem2}, we have $d_1(K \# (J_{n,1})^*) = -2V_0(K \# (J_{n,1})^*) \neq 0$.
Furthermore,  by \cite[Theorem 1.2]{hedden}, we have
$$
\tau(K \# (J_{n,1})^*) = \tau(K) - \tau(J_{n,1}) \leq \tau(K) - n\cdot\tau(J)<0.
$$
Hence it follows from Proposition \ref{prop4} that $K \# (J_{n,1})^*$ is not slice in any definite 4-manifold.
\end{proof}

\def\proofname{Proof of Theorem \ref{thm1}}
\begin{proof}
For a knot $K$, let $\Wh(K)$ denote the positively-clasped untwisted Whitehead double of $K$. Then we set
$$K_n:= (\#_{i=1}^{3} \Wh(T_{2,3})) \# ((\Wh(T_{2,3}))_{n+3,1})^*$$
for any positive integer $n$.
Since the Alexander polynomial of $K_n$ equals 1, $K_n$ is topologically slice for any $n$.
\cite{freedman, freedman-quinn}

We first prove that $K_n$ is not slice in any definite 4-manifold. By Lemma \ref{q=1} and Lemma \ref{lem4}, we have
$$
V_0(K_n) \geq V_0(\#_{i=1}^3 \Wh(T_{2,3})) -V_0(\Wh(T_{2,3}))
$$
for any $n$.
As mentioned in \cite[Section 3]{kim-park}, it is proved in \cite[Theorem B.1]{hedden-kim-livingston} that $\#_{i=1}^k \Wh(T_{2,3})$ is $\nu^+$-equivalent to $T_{2,2k+1}$ for any $k>0$.
(Here, knots $K$ and $J$ being $\nu^+$-equivalent means that the equalities 
$\nu^+(K\#J^*) =\nu^+(J \# K^*)=0$ hold.)
Hence it follows from \cite[Lemma 3.1]{kim-park}, the definition of $V_k$ and 
\cite[Corollary 1.5]{ozsvath-szabo2} that
$V_0(\#_{i=1}^k \Wh(T_{2,3})) = V_0(T_{2,2k+1}) = \lceil \frac{k}{2} \rceil$. This implies that
$$
V_0(K_n) \geq \Big\lceil \frac{3}{2} \Big\rceil - \Big\lceil \frac{1}{2} \Big\rceil =1>0.
$$
In particular, we have $d_1(K_n) \neq 0$.
Moreover, it follows from \cite[Theorem 1.5]{hedden2} that $\tau(\Wh(T_{2,3})) =1$,
and hence \cite[Theorem 1.2]{hedden} implies that
$$
\tau(K_n) = 3 - (n+3) = -n <0.
$$
Therefore, it follows from Proposition \ref{prop4} that for any $n>0$, the knot $K_n$ is not slice in any definite 4-manifold.

Next we prove that the knots $\{K_n\}_{n \in \mathbb{Z}_{>0}}$ are linearly independent.
Suppose that a linear combination $m_1[K_{n_1}]+ \cdots + m_k[K_{n_k}]$ equals zero in $\mathcal{C}$ (where we may assume that $0<n_1 < n_2 < \ldots < n_k$). Then we have the equality
\begin{equation}
\label{eq2}
3(\Sigma_{i=1}^{k}m_i)[\Wh(T_{2,3})] =m_1[(\Wh(T_{2,3}))_{n_1+3,1}] + \cdots 
+m_k [(\Wh(T_{2,3}))_{n_k+3,1}]. 
\end{equation}
In the proof of \cite[Therorem A]{kim-park}, the authors define a homomorphism 
$\phi: \mathcal{C} \to \mathbb{Z}^{\infty}$ and show that 
$\phi([(\Wh(T_{2,3}))_{n+3,1}])=(*, \cdots, *, 1, 0, 0, \cdots)$ where 1 is the 
$(n+2)^{\text{nd}}$ coordinate.
Hence we can see that the 
$(n_k +2)^{\text{nd}}$ coordinate of $\phi(\text{RHS of (\ref{eq2})})$ is $m_k$.
On the other hand, we can verify that 
$\phi ([\Wh(T_{2,3})]) = \phi([T_{2,3}]) = (0,0, \cdots)$, and hence
$m_k$ must be 0. Inductively, we have $m_1 = \cdots = m_k = 0$.
This completes the proof.
\end{proof}

\begin{remark}
If we do not require knots in Theorem \ref{thm1} to be topologically slice, then the existence of such a family can be established using the following proposition, which immediately follows from \cite[Proposition 1.2]{cochran-harvey-horn}.

\begin{prop}
If the Levine-Tristram signature of a knot $K$ has both positive and negative values,
then $K$ is not smoothly slice in any definite 4-manifold.
\end{prop}

Indeed, we can take $J_k:=\{ T_{2,2k+9}\#(\#_{i=1}^{k+5}T_{2,3})^* \}_{k \in \mathbb{Z}_{>0}}$
as the concrete sequence.
(we can verify that $\sigma_{J_k}(e^{i\theta})= -2$ for 
$\theta \in(\frac{\pi }{2k+9}, \frac{3\pi }{2k+9})$ and 
$\sigma_{J_k}(-1)= 2$.
Furthermore, since all torus knots are linearly independent in $\mathcal{C}$ \cite{litherland},
the knots $J_k$ are also linearly independent.)
\end{remark}

Finally we prove Theorem \ref{thm2}. To do so, we use 
the following observation relating kinkiness to $\nu^+$ and $\tau$.

\begin{lem}
\label{lem7}
For any knot $K$, we have the inequalities
$$
\nu^+(K) \leq k^+(K)
$$
and
$$
-k^-(K) \leq \tau(K) \leq k^+(K). 
$$
\end{lem}

\def\proofname{Proof}
\begin{proof}
If a knot  $K_1$ is deformed into $K_2$ by a crossing change 
from a positive crossing (Figure \ref{pos}) to a negative crossing (Figure \ref{neg})
(resp.\ from a negative crossing to a positive crossing),
then we say that $K_1$ is deformed into $K_2$
by a {\it positive (resp.\ negative) crossing change}. 
It is proved in \cite[Theorem 1.3]{bodnar-celoria-golla} and \cite[Corollary 1.5]{ozsvath-szabo3} that if a knot  $K_+$ is deformed into $K_-$ by a positive crossing change,
then we have
$$
\nu^+(K_-) \leq \nu^+ (K_+) \leq \nu^+ (K_-) + 1
$$
and
$$
\tau(K_-) \leq \tau(K_+) \leq \tau(K_-)+1.
$$
Furthermore, it follows from \cite[Proposition 2.1]{owens-strle}
that for any knot $K$, there exists a knot $J$ so that $J$ is concordant to $K$ and $J$ can be deformed into a slice knot $L$ by just $k^+(K)$ positive crossing changes 
and finitely many negative crossing changes.
These imply that 
$$
\nu^+(K) = \nu^+(J) \leq \nu^+(L) + k^+(K) = k^+(K)
$$
and
$$
\tau(K) = \tau(J) \leq \tau(L) + k^+(K) = k^+(K).
$$

By applying the same argument to $K^*$, we have
$$
-\tau(K)=\tau(K^*) \leq k^+(K^*) = k^-(K).
$$
\end{proof}

\begin{figure}[htbp]
\begin{minipage}[]{0.4\hsize}
\hspace{-1mm}
\includegraphics[scale = 0.7]{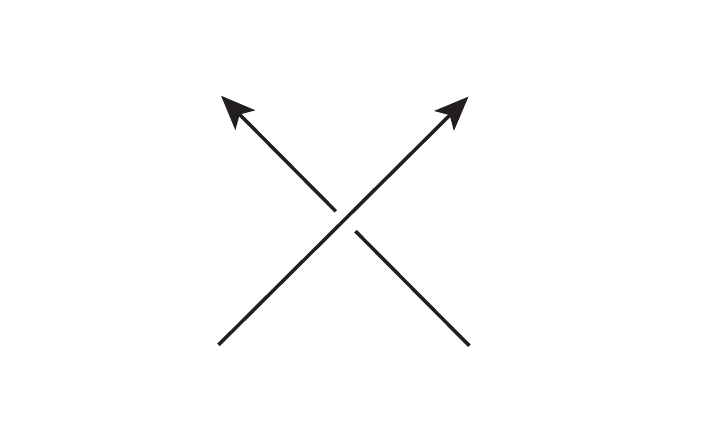}
\caption{\label{pos}}
 \end{minipage}
 \begin{minipage}[]{0.4\hsize}
\hspace{-4mm}
\includegraphics[scale = 0.7]{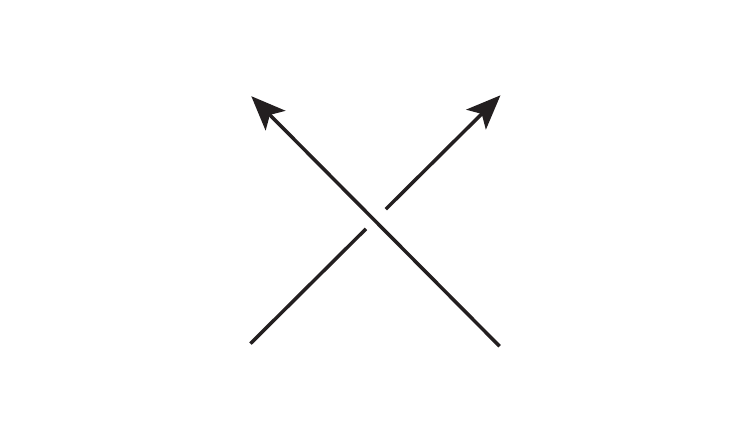}
\caption{\label{neg}}
 \end{minipage}
\end{figure}

\def\proofname{Proof of Theorem \ref{thm2}}
\begin{proof}
For positive integers $k$ and $l$, we define $K_{k,l}$ by 
$$
K_{k,l}:= (\#_{i=1}^{2k+1}\Wh (T_{2,3})) \# ((\Wh(T_{2,3}))_{l+2k+1,1})^*.
$$
Obviously, $K_{k,l}$ is topologically slice for any $k,l>0$.
We prove that for any $m,n \in \mathbb{Z}_{>0}$,  $\{K_{m,l}\}_{l \geq n}$ are mutually distinct  
and each of them satisfies $k^+(K_{m,l}) \geq m$ and $k^-(K_{m,l}) \geq n$.

By applying the argument in the proof of Theorem \ref{thm1}, we have
$$
V_0(K_{k,l})\geq V_0(\#_{i=1}^{2k+1}\Wh (T_{2,3})) - V_0((\Wh(T_{2,3}))_{l+2k+1,1})=k
$$
and
$$
\tau(K_{k,l}) = 2k+1 - (l+2k+1) = -l.
$$
In particular, $K_{k,l}$ is not equal to $K_{k,l'}$ if $l \neq l'$.
Furthermore, since $\nu^+(K) = \min \{i\in \mathbb{Z}_{\geq0}|V_i(K)=0\}$ and 
$V_{i+1}(K) \geq V_i(K) -1$,
we have $\nu^+(K_{k,l}) \geq k$.
Hence Lemma \ref{lem7} proves that $k^+(K_{k,l}) \geq k$ and $k^-(K_{k,l}) \geq l$.
This completes the proof.
\end{proof}


\begin{thebibliography}{9}
\bibitem{bodnar-celoria-golla}
J.\ Bodn\'{a}r, D.\ Celoria and M.\ Golla,
{\it A note on cobordisms of algebraic knots}.
arXiv:1509.08821 (2015).

\bibitem{cochran-harvey-horn}
T.\ D.\ Cochran, S.\ Harvey and P.\ Horn,
{\it Filtering smooth concordance classes of topologically slice knots}. 
Geom.\ Topol. 17 (2013), no.\ 4, 2103--2162. 

\bibitem{endo}
H.\ Endo,
{\it Linear independence of topologically slice knots in the smooth cobordism group.}
Topology Appl. 63 (1995), no.\ 3, 257--262. 

\bibitem{freedman}
M.\ Freedman, 
{\it The topology of four-dimensional manifolds.} 
J. Differential Geom. 17 (1982), no. 3, 357--453. 

\bibitem{freedman-quinn}
M.\ Freedman and F.\ Quinn,
{\it Topology of 4-manifolds.} 
Princeton Mathematical Series, 39. Princeton University Press, Princeton, NJ, 1990.

\bibitem{gompf}
R.\ E.\ Gompf, 
{\it Smooth concordance of topologically slice knots}. 
Topology 25 (1986), no.\ 3, 353--373. 

\bibitem{gompf-stipsicz}
R.\ E.\ Gompf and A.\ I.\ Stipsicz,
{\it 4-manifolds and Kirby calculus.} 
Graduate Studies in Mathematics, 20. American Mathematical Society, Providence, RI, 1999.

\bibitem{hedden2} M.\ Hedden,
{\it Knot Floer homology of Whitehead doubles}.
Geom. Topol. 11 (2007), 2277--2338.

\bibitem{hedden} M.\ Hedden,
{\it On knot Floer homology and cabling. II}.
Int. Math. Res. Not. IMRN 2009, no.\ 12, 2248--2274. 

\bibitem{hedden-kim-livingston}
M.\ Hedden, S.\ G.\ Kim and C.\ Livingston,
{\it Topologically slice knots of smooth concordance order two}.
J. Differential Geom. 102 (2016), no.\ 3, 353--393. 

\bibitem{hom-wu} J.\ Hom and Z.\ Wu,
\textit{Four-ball genus bounds and a refinement of the
Ozsv\'{a}th-Szab\'{o} tau-invariant}.
arXiv:1401.1565 (2014).

\bibitem{kim-park}
M.\ H.\ Kim and K.\ Park,
{\it An infinite-rank summand of knots with trivial Alexander polynomial}.
arXiv:1604.04037 (2016).

\bibitem{litherland}
R.\ A.\ Litherland, 
{\it Signatures of iterated torus knots.} 
Topology of low-dimensional manifolds (Proc.\ Second Sussex Conf., Chelwood Gate, 1977), pp. 71--84, 
Lecture Notes in Math., 722, Springer, Berlin, 1979. 

\bibitem{ni-wu}
Y.\ Ni and Z.\ Wu,
{\it Cosmetic surgeries on knots in $S^3$}. 
J.\ Reine\ Angew.\ Math. 706 (2015), 1--17. 

\bibitem{owens-strle}
B.\ Owens and S.\ Strle,
{\it Immersed disks, slicing numbers and concordance unknotting numbers}.
arXiv:1311.6702 (2013).

\bibitem{ozsvath-szabo} P.\ Ozsv\'{a}th and Z.\ Szab\'{o},
\textit{Absolutely graded Floer homologies
and intersection forms for four-manifolds with boundary}.
Adv. Math.
173 (2003), 179--261.

\bibitem{ozsvath-szabo2} P.\ Ozsv\'{a}th and Z.\ Szab\'{o},
{\it Heegaard Floer homology and alternating knots}.
Geom.\ Topol. 7 (2003), 225--254.

\bibitem{ozsvath-szabo4} 
P.\ Ozsv\'{a}th and Z.\ Szab\'{o},
{\it Knot Floer homology and rational surgeries}. 
Algebr.\ Geom.\ Topol. 11 (2011), no. 1, 1--68. 

\bibitem{ozsvath-szabo3} P.\ Ozsv\'{a}th and Z.\ Szab\'{o},
{\it Knot Floer homology and the four-ball genus}.
Geom.\ Topol. 7 (2003), 615--639. 

\bibitem{peters} T.\ D.\ Peters,
\textit{A concordance invariant from the Floer homology
of $\pm1$ surgeries}.
arXiv:1003.3038 (2010).

\bibitem{rasmussen} J.\ Rasmussen,
{\it Floer homology and knot complements}.
arXiv:0306.378 (2003).

\bibitem{wu} Z.\ Wu,
\textit{ A cabling formula for $\nu^+$ invariant}.
arXiv:1501.04749 (2015).

\end{thebibliography}
\end{document}